\documentclass[12pt,bezier]{article}
\usepackage{times}
\usepackage{booktabs}
\usepackage{pifont}
\usepackage{floatrow}
\usepackage{caption}
\usepackage{mathrsfs}
\usepackage[fleqn]{amsmath}
\usepackage{amsfonts,amsthm,amssymb,mathrsfs,bbding}
\usepackage{txfonts}
\usepackage{graphics,multicol}
\usepackage{graphicx}
\usepackage{color}
\usepackage{enumerate}
\usepackage{caption}
\usepackage{cite}
\usepackage{latexsym,bm}
\usepackage{indentfirst}
\usepackage{mathtools}
\evensidemargin=\oddsidemargin\topmargin=-1.5cm

\newtheorem{thm}{Theorem}[section]

\newtheorem{lem}{Lemma}[section]
\newtheorem{cor}{Corollary}[section]

\newtheorem{remark}{Remark}

\theoremstyle{definition}

\addtocounter{section}{0}

\begin{document}
\title{On regular graphs with four distinct eigenvalues\footnote{This work is supported
by NSFC (Grant Nos. 11671344, 11261059 and 11531011).}}
\author{{\small Xueyi Huang, \ \ Qiongxiang Huang\footnote{
Corresponding author.}\setcounter{footnote}{-1}\footnote{
\emph{E-mail address:} huangxymath@gmail.com (X. Huang), huangqx@xju.edu.cn (Q. Huang).} }\\[2mm]\scriptsize
College of Mathematics and Systems Science,
\scriptsize Xinjiang University, Urumqi, Xinjiang 830046, P. R. China}

\date{}
\maketitle
{\flushleft\large\bf Abstract}
Let $\mathcal{G}(4,2)$ be the set of connected regular graphs with four distinct eigenvalues in which exactly two eigenvalues are simple,  $\mathcal{G}(4,2,-1)$ (resp. $\mathcal{G}(4,2,0)$) the set of graphs belonging to $\mathcal{G}(4,2)$ with $-1$ (resp. $0$) as an eigenvalue, and $\mathcal{G}(4,\geq -1)$ the set of connected regular graphs with four distinct eigenvalues and  second least eigenvalue  not less than $-1$. In this paper, we prove the non-existence of connected graphs having four distinct eigenvalues in which at least three eigenvalues are simple, and determine all the graphs in  $\mathcal{G}(4,2,-1)$.  As a by-product of this work,  we   characterize all the graphs belonging to $\mathcal{G}(4,\geq-1)$ and $\mathcal{G}(4,2,0)$, respectively, and   show that all these graphs  are determined by their spectra.

\vspace{0.1cm}
\begin{flushleft}
\end{flushleft}

\section{Introduction}\label{s-1}
Let $G=(V(G), E(G))$ be a simple  undirected graph  on $n$ vertices with adjacency matrix $A=A(G)$. Denote by $\lambda_1,\lambda_2,\ldots,\lambda_t$  all the distinct eigenvalues of $A$ with multiplicities $m_1,m_2,\ldots,m_t$ ($\sum_{i=1}^tm_i=n$), respectively. These eigenvalues are also called the \textit{eigenvalues} of $G$. All the eigenvalues together with their multiplicities are called the \emph{spectrum} of $G$ denoted by $\mathrm{Spec}(G)=\big\{[\lambda_1]^{m_1},[\lambda_2]^{m_2},\ldots,[\lambda_t]^{m_t}\big\}$.  If $G$ is a connected $k$-regular graph, then $\lambda_1$ denotes $k$, and has multiplicity $m_1=1$.

A graph $G$ is said to be \emph{determined by its  spectrum} (DS for short) if  $G\cong H$ whenever $\mathrm{Spec}_A(G)=\mathrm{Spec}_A(H)$ for any graph $H$. Also, a graph $G$ is called \emph{walk-regular} if for which the number of walks of length $r$ from a given vertex $x$ to itself (closed walks) is independent of the choice of $x$ for all $r$ (see \cite{Godsil}). Note that a walk-regular graph is always regular, but in general the converse is not true.

Throughout this paper,  we denote  the \emph{neighbourhood} of a vertex $v\in V(G)$ by $N_{G}(v)$,  the complete graph on  $n$ vertices by $K_n$,  the complete multipartite graph with $s$ parts of sizes $n_1,\ldots,n_s$  by $K_{n_1,\ldots,n_s}$, and the graph obtained by removing a perfect matching from $K_{n,n}$ by $K_{n,n}^-$. Also,  the  $n\times n$ identity matrix,  the $n\times 1$ all-ones  vector and  the $n\times n$ all-ones matrix will
be denoted by  $I_n$, $\mathbf{e}_n$  and $J_n$, respectively.

Connected graphs with a few  eigenvalues have aroused  a lot of  interest in the past several decades.  This problem was perhaps first raised by Doob \cite{Doob}. It is well known that connected regular graphs having three distinct eigenvalues are  strongly regular graphs \cite{Shrikhande}, and connected regular bipartite graphs having four distinct eigenvalues are the incidence graphs of  symmetric balanced incomplete block designs \cite{Brouwer,Cvetkovic}. Furthermore, connected non-regular graphs with three distinct eigenvalues and least eigenvalue $-2$ were determined by Van Dam \cite{Dam3}.  Very recently, Cioab\u{a} et al.  in  \cite{Cioaba1} (resp. \cite{Cioaba})  determined all graphs with at most two eigenvalues (multiplicities included) not equal to $\pm1$ (resp. $-2$ or $0$). De Lima et al. in \cite{Lima} determined all connected non-bipartite graphs with all but two eigenvalues in the interval $[-1, 1]$. For more results on graphs with few eigenvalues, we refer the reader to \cite{Bridges,Caen,Cheng,Dam1,Dam2,Dam4,Dam5,Dam7,Doob,Muzychuk,Rowlinson}.

Van Dam in \cite{Dam1,Dam2} investigated the connected regular graphs with four distinct eigenvalues. He classified such graphs into three classes according to the number of integral eigenvalues  (see Lemma \ref{lem-1} below). Based on Van Dam's classification and the number of simple eigenvalues, we can classify such graphs more precisely, that is, if $G$ is a connected $k$-regular graphs with four distinct eigenvalues, then
\begin{enumerate}[(1)]
\vspace{-0.2cm}
\item $G$ has at least three  simple eigenvalues, or
\vspace{-0.25cm}
\item $G$ has two simple eigenvalues:
\vspace{-0.25cm}
\begin{enumerate}[(2a)]
\item $G$ has four integral eigenvalues in which two eigenvalues are simple;
\vspace{-0.2cm}
\item $G$ has two integral eigenvalues, which are simple, and two eigenvalues of the form $\frac{1}{2}(a\pm\sqrt{b})$, with $a,b\in \mathbb{Z}$, $b>0$, with the same multiplicity, or
\end{enumerate}
\vspace{-0.35cm}
\item $G$ has one simple eigenvalue, i.e., its degree $k$:
\vspace{-0.25cm}
\begin{enumerate}[(3a)]
\item $G$ has four integral eigenvalues;
\vspace{-0.2cm}
\item $G$ has two integral eigenvalues, and two eigenvalues of the form $\frac{1}{2}(a\pm\sqrt{b})$, with $a,b\in \mathbb{Z}$, $b>0$, with the same multiplicity;
    \vspace{-0.2cm}
\item $G$ has one integral eigenvalue, its degree $k$, and the other three have the same multiplicity $m=\frac{1}{3}(n-1)$, and $k=m$ or $k=2m$.
\end{enumerate}
   \vspace{-0.3cm}
\end{enumerate}

In this paper, we continue to focus on connected regular  graphs with four distinct eigenvalues. Concretely, we  show that there are no graphs in (1), and   give a complete characterization of the graphs belonging to $\mathcal{G}(4,2,-1)$: if $-1$ is a non-simple eigenvalue, we determine all such graphs; if $-1$ is a simple eigenvalue,  we prove that such graphs cannot belong to (2a) and (2b), respectively,  and so do not exist. In the process, we determine all the graphs  in $\mathcal{G}(4,\geq -1)$ and $\mathcal{G}(4,2,0)$, respectively, and show that all these graphs  are DS.

\section{Main tools}\label{s-2}
In this section, we recall some results from the literature  that will be useful in the next section.

\begin{lem}\label{lem-1} (See \cite{Dam1,Dam2}.)
If $G$ is a connected $k$-regular graph on $n$ vertices
with four distinct eigenvalues, then
\begin{enumerate}[(i)]
\vspace{-0.3cm}
\item $G$ has four integral eigenvalues, or
\vspace{-0.3cm}
\item $G$ has two integral eigenvalues, and two eigenvalues of the form $\frac{1}{2}(a\pm\sqrt{b})$, with $a,b\in \mathbb{Z}$, $b>0$, with the same multiplicity, or
\vspace{-0.3cm}
\item $G$ has one integral eigenvalue, its degree $k$, and the other three have the same multiplicity $m=\frac{1}{3}(n-1)$, and $k=m$ or $k=2m$.
\end{enumerate}
\end{lem}

\begin{lem}\label{lem-5}(See \cite{Brouwer}.)
If $G$ is connected and regular with four distinct eigenvalues,
then $G$ is walk-regular.
\end{lem}

Let $G$ be a $k$-regular graph. We say that $G$ admits a \emph{regular partition into halves} with degrees $(a,b)$ ($a+b=k$) if we can partition the vertices of $G$ into two parts of equal size such that every vertex has $a$ neighbors in its own part and $b$ neighbors in the other part \cite{Dam1}.
\begin{lem}\label{lem-2}(See \cite{Dam1}.)
Let $G$ be a connected walk-regular graph on $n$ vertices
and degree $k$, having distinct eigenvalues $k,\lambda_2,\lambda_3,\ldots,\lambda_t$,
of which an eigenvalue unequal to $k$, say $\lambda_j$, has multiplicity $1$.
Then $n$ is even and $G$ admits a regular partition into halves with degrees
$(\frac{1}{2}(k+\lambda_j),\frac{1}{2}(k-\lambda_j))$. Moreover, $n$ is a divisor of
$$\prod_{i\neq j}(k-\lambda_i)+\prod_{i\neq j}(\lambda_j-\lambda_i)~~and~~\prod_{i\neq j}(k-\lambda_i)-\prod_{i\neq j}(\lambda_j-\lambda_i).$$
\end{lem}

From the proof of Lemma \ref{lem-2}, we  obtain  the following corollary immediately.

\begin{cor}\label{cor-1}
Under the assumption of Lemma \ref{lem-2},  the eigenvector of $\lambda_j$ can be written as $\mathbf{x}_j=\frac{1}{\sqrt{n}}(\mathbf{e}_{\frac{n}{2}}^T,-\mathbf{e}_{\frac{n}{2}}^T)^T$, and the vertex partition $V(G)=V_1\cup V_2$ with $V_1=\{v\in V\mid \mathbf{x}_{j}(v)=1\}$ and $V_2=\{v\in V\mid \mathbf{x}_{j}(v)=-1\}$ is just the regular partition of $G$ into halves with degrees $(\frac{1}{2}(k+\lambda_j),\frac{1}{2}(k-\lambda_j))$ described in Lemma \ref{lem-2}.
\end{cor}
A \emph{balanced incomplete block design}, denoted by \emph{BIBD}, consists of $v$ elements and $b$ subsets of these elements called \emph{blocks} such that each element is contained
in $t$ blocks,  each block contains $k$ elements, and  each pair of elements is simultaneously contained in $\lambda$ blocks (see \cite{Cvetkovic}). The integers $(v,b,r,k,\lambda)$ are called the \emph{parameters} of the design. In the case $r=k$ (and then $v=b$) the design is called \emph{symmetric} with parameters $(v,k,\lambda)$.

The \emph{incidence graph} of a $BIBD$ is the bipartite graph on  $b+v$ vertices (correspond to the blocks and elements of the design) with two vertices adjacent if and only if one corresponds to a block and the other corresponds to an element contained in that block. As   shown in \cite{Cvetkovic},  the incidence graph has spectrum $\big\{[\sqrt{rk}]^1,[\sqrt{r-\lambda}]^{v-1},[0]^{b-v},[-\sqrt{r-\lambda}]^{v-1},$ $[-\sqrt{rk}]^1\big\}$. In particular, if the design is symmetric,  then the incidence graph is a $k$-regular bipartite graph with spectrum $\big\{[k]^1,[\sqrt{k-\lambda}]^{v-1},[-\sqrt{k-\lambda}]^{v-1},[-k]^1\big\}$.

The following lemma gives a characterization of regular bipartite graphs with four distinct eigenvalues.
\begin{lem}\label{lem-4}(See \cite{Brouwer,Cvetkovic}.)
A connected regular bipartite  graph $G$ with four distinct eigenvalues
is the incidence graph of a symmetric BIBD.
\end{lem}

\begin{figure}[t]
\centering
\begin{center}
\unitlength 1.75mm 
\linethickness{0.8pt}
\ifx\plotpoint\undefined\newsavebox{\plotpoint}\fi 
\begin{picture}(52,27)(0,0)
\put(2,11){\circle*{1}}
\put(10,11){\circle*{1}}
\put(18,11){\circle*{1}}
\put(2,5){\circle*{1}}
\put(10,5){\circle*{1}}
\put(18,5){\circle*{1}}
\put(2,8){\oval(4,10)[]}
\put(10,8){\oval(4,10)[]}
\put(18,8){\oval(4,10)[]}
\put(25,8){\circle*{1}}
\put(20,8){\line(1,0){5}}
\put(4,8){\line(1,0){4}}
\put(12,8){\line(1,0){4}}
\put(36,11){\circle*{1}}
\put(44,11){\circle*{1}}
\put(36,5){\circle*{1}}
\put(44,5){\circle*{1}}
\put(36,8){\oval(4,10)[]}
\put(44,8){\oval(4,10)[]}
\put(38,8){\line(1,0){4}}
\put(36,13){\line(0,1){4}}
\put(44,13){\line(0,1){4}}
\put(36,25){\circle*{1}}
\put(44,25){\circle*{1}}
\put(36,19){\circle*{1}}
\put(44,19){\circle*{1}}
\put(36,22){\oval(4,10)[]}
\put(44,22){\oval(4,10)[]}
\put(38,22){\line(1,0){4}}
\put(46,22){\line(1,0){5}}
\put(46,8){\line(1,0){5}}
\put(51,22){\circle*{1}}
\put(51,8){\circle*{1}}
\put(2.1,8.5){\makebox(0,0)[cc]{$\vdots$}}
\put(10.1,8.5){\makebox(0,0)[cc]{$\vdots$}}
\put(18.1,8.5){\makebox(0,0)[cc]{$\vdots$}}
\put(36.1,22.5){\makebox(0,0)[cc]{$\vdots$}}
\put(36.1,8.5){\makebox(0,0)[cc]{$\vdots$}}
\put(36.1,22.5){\makebox(0,0)[cc]{$\vdots$}}
\put(44.1,8.5){\makebox(0,0)[cc]{$\vdots$}}
\put(44.1,22.5){\makebox(0,0)[cc]{$\vdots$}}
\put(1,11){\makebox(0,0)[cc]{\scriptsize$1$}}
\put(1,5){\makebox(0,0)[cc]{\scriptsize$l$}}
\put(9,11){\makebox(0,0)[cc]{\scriptsize$1$}}
\put(9,5){\makebox(0,0)[cc]{\scriptsize$m$}}
\put(17,11){\makebox(0,0)[cc]{\scriptsize$1$}}
\put(17,5){\makebox(0,0)[cc]{\scriptsize$n$}}
\put(35,11){\makebox(0,0)[cc]{\scriptsize$1$}}
\put(35,5){\makebox(0,0)[cc]{\scriptsize$m$}}
\put(43,11){\makebox(0,0)[cc]{\scriptsize$1$}}
\put(43,5){\makebox(0,0)[cc]{\scriptsize$p$}}
\put(35,25){\makebox(0,0)[cc]{\scriptsize$1$}}
\put(35,19){\makebox(0,0)[cc]{\scriptsize$l$}}
\put(43,25){\makebox(0,0)[cc]{\scriptsize$1$}}
\put(43,19){\makebox(0,0)[cc]{\scriptsize$n$}}
\put(11,0){\makebox(0,0)[cc]{\footnotesize$A(l,m,n)$}}
\put(41,0){\makebox(0,0)[cc]{\footnotesize$B(l,m,n,p)$}}
\end{picture}
\end{center}
\vspace{-0.4cm}
\caption{\footnotesize{The graphs $A(l,m,n)$ and $B(l,m,n,p)$.}}\label{fig-1}
\end{figure}
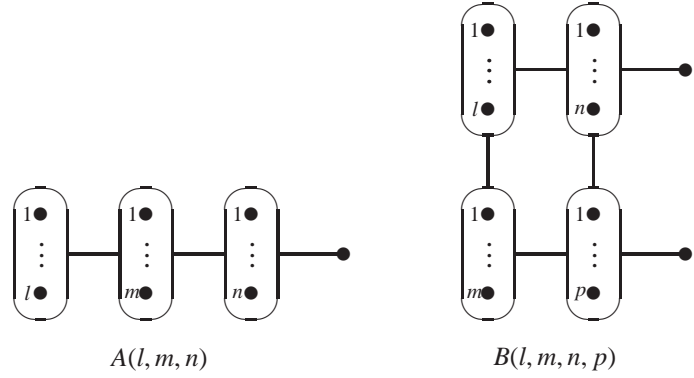
Denote by $A(l,m,n)$ and $B(l,m,n,p)$ ($l,m,n,p\geq 1$) the two graphs from Fig.\ref{fig-1}, where the vertices contained in an ellipse form an independent set, and any two ellipses or a vertex and an ellipse joined with one  line denote a complete bipartite graph.
\begin{lem}\label{lem-6} (See \cite{Torgasev}.)
The second least eigenvalue of a connected graph $G$ is  greater than $-1$ if and only if
\begin{enumerate}[(i)]
\vspace{-0.3cm}
\item $G=K_{m,n}$ with $m,n\geq 1$, or
\vspace{-0.3cm}
\item $G=A(l,m,n)$ with one of the following cases occours: $n=1$ and $m,l\geq 1$;  $n=l=2$ and $m\geq 1$;  $n\geq 2$ and $l=m=1$; $n\geq 2$, $l=1$ and $m\geq 1$, or
\vspace{-0.3cm}
\item $G=B(l,m,n,p)$ with $(p+l-pl)(m+n-mn)>(p-1)(n-1)$.
\end{enumerate}
\end{lem}

\section{Main results}\label{s-3}

Let $A$ be a real symmetric matrix whose all distinct eigenvalues are $\lambda_1,\ldots,\lambda_s$. Then $A$ has the \emph{spectral decomposition} $A=\lambda_1P_1+\cdots+\lambda_sP_s$, where $P_i=\mathbf{x}_1\mathbf{x}_1^T+\cdots+\mathbf{x}_{d}\mathbf{x}_{d}^T$ if the eigenspace $\mathcal{E}(\lambda_i)$ has $\{\mathbf{x}_1,\ldots,\mathbf{x}_d\}$ as an orthonormal basis. It is easy to see that  for any real polynomial $f(x)$, we have $f(A)=f(\lambda_1)P_1+\cdots+f(\lambda_s)P_s$.

First of all, we will prove the non-existence of connected regular graphs with four distinct eigenvalues in which at least three eigenvalues are simple.

\begin{thm}\label{thm-main-0}
There are no connected  $k$-regular graphs on $n$ ($n\geq 4$) vertices with spectrum $\big\{[k]^1,[\lambda_2]^1,[\lambda_3]^1,[\lambda_4]^{n-3}\big\}$.
\end{thm}

\begin{proof}
Suppose that $G$ is a connected $k$-regular graph on $n$ vertices  with adjacency matrix $A$ and spectrum $\big\{[k]^1,[\lambda_2]^1,$ $[\lambda_3]^1,[\lambda_4]^{n-3}\big\}$. Then $G$ has minimal polynomial $p(x)=(x-k)(x-\lambda_2)(x-\lambda_3)(x-\lambda_4)$. By Lemmas \ref{lem-5} and \ref{lem-2}, $\lambda_2$ and $\lambda_3$ are integers,  so $\lambda_4$ is also an integer. Furthermore, by Corollary \ref{cor-1}, we may assume that $\lambda_2$ and $\lambda_3$, respectively, have orthonormal eigenvectors as follows:
\begin{equation*}
\mathbf{x}_2=\frac{1}{\sqrt{n}}(\mathbf{e}_\frac{n}{2}^T,-\mathbf{e}_\frac{n}{2}^T)^T\ \mbox{ and }\ \mathbf{x}_3=\frac{1}{\sqrt{n}}(\mathbf{e}_\frac{n}{4}^T,-\mathbf{e}_\frac{n}{4}^T,\mathbf{e}_\frac{n}{4}^T,-\mathbf{e}_\frac{n}{4}^T)^T.
\end{equation*}
Taking $f(x)=x-\lambda_4$, by the spectral decomposition of $f(A)$ we get
\begin{equation*}
A-\lambda_4I_n=\frac{1}{n}(k-\lambda_4)\mathbf{e}_n\mathbf{e}_n^T+(\lambda_2-\lambda_4)\mathbf{x}_2\mathbf{x}_2^T+(\lambda_3-\lambda_4)\mathbf{x}_3\mathbf{x}_3^T,
\end{equation*}
or equivalently,
\begin{equation}\label{equ-1}
n(A-\lambda_4I_n)=(k-\lambda_4)J_n+(\lambda_2-\lambda_4)
\left(\begin{matrix}
J_\frac{n}{2}&-J_\frac{n}{2}\\
-J_\frac{n}{2}&J_\frac{n}{2}
\end{matrix}\right)
+(\lambda_3-\lambda_4)
\left(\begin{smallmatrix}J_\frac{n}{4}&-J_\frac{n}{4}&J_\frac{n}{4}&-J_\frac{n}{4}\\
-J_\frac{n}{4}&J_\frac{n}{4}&-J_\frac{n}{4}&J_\frac{n}{4}\\
J_\frac{n}{4}&-J_\frac{n}{4}&J_\frac{n}{4}&-J_\frac{n}{4}\\
-J_\frac{n}{4}&J_\frac{n}{4}&-J_\frac{n}{4}&J_\frac{n}{4}\\
\end{smallmatrix}\right).
\end{equation}
On the other hand, by considering the traces of $A$ and $A^2$, respectively, we obtain
\begin{align}
&k+\lambda_2+\lambda_3+(n-3)\lambda_4=0,\label{equ-2}\\
&k^2+\lambda_2^2+\lambda_3^2+(n-3)\lambda_4^2=kn.\label{equ-3}
\end{align}

Now we partition $V(G)$ the same way as we partition the matrix $\mathbf{x}_3\mathbf{x}_3^T$ in (\ref{equ-1}), and denote by $V_1,V_2,V_3,V_4$ the corresponding vertex subsets, respectively. By considering the block matrix $A(V_1,V_4)$ in (\ref{equ-1}),  we have
\begin{equation}\label{equ-4}
nA(V_1,V_4)=((k-\lambda_4)-(\lambda_2-\lambda_4)-(\lambda_3-\lambda_4))J_{\frac{n}{4}}.
\end{equation}

First suppose that $\lambda_4=0$. From (\ref{equ-2}) and (\ref{equ-4}) we know that $nA(V_{1},V_{4})=(k-\lambda_2-\lambda_3)J_{\frac{n}{4}}=2kJ_{\frac{n}{4}}$. Thus $n=2k$ and $A(V_{1},V_{4})=J_{\frac{n}{4}}$ because $k\neq 0$. Putting $n=2k$ in (\ref{equ-3}), we get $\lambda_2^2+\lambda_3^2=k^2$, and so $\lambda_2\lambda_3=0$ by (\ref{equ-2}). This implies that $\lambda_2=0$ or $\lambda_3=0$, which is a contradiction because $\lambda_2$, $\lambda_3$ and $\lambda_4$ are distinct.

Now we can assume that  $\lambda_4\neq0$. For the block matrix $A(V_{1},V_{1})$,  from (\ref{equ-1}) and (\ref{equ-2})   it is seen that
$n(A(V_{1},V_{1})-\lambda_4I_{\frac{n}{4}})=((k-\lambda_4)+(\lambda_2-\lambda_4)+(\lambda_3-\lambda_4))J_{\frac{n}{4}}=-n\lambda_4J_{\frac{n}{4}}$,
that is, $A(V_{1},V_{1})=-\lambda_4(J_{\frac{n}{4}}-I_{\frac{n}{4}})$.
If $n\geq 8$, then $J_{\frac{n}{4}}-I_{\frac{n}{4}}\neq 0$. Thus $\lambda_4=-1$  because $\lambda_4\neq 0$. Putting $\lambda_4=-1$ in (\ref{equ-4}), we get $nA(V_{1},V_{4})=(k-\lambda_2-\lambda_3-1)J_{\frac{n}{4}}$. Then $A(V_{1},V_{4})=0$ or $A(V_{1},V_{4})=J_{\frac{n}{4}}$. If $A(V_{1},V_{4})=0$, we have $k-\lambda_2-\lambda_3-1=0$. Then from (\ref{equ-2}) and (\ref{equ-3}), we get $\lambda_2=k$ and $\lambda_3=-1$, or $\lambda_2=-1$ and $\lambda_3=k$, a contradiction.
Thus $A(V_{1},V_{4})=J_{\frac{n}{4}}$, and so  $k-\lambda_2-\lambda_3-1=n$. Again from (\ref{equ-2}), we get $k=n-1$, which implies that $G$ is a complete graph, a contradiction. If $n<8$, from the above arguments we see that $\frac{n}{4}$ is an integer, so $n=4$. Then $G=K_4$ or $G=C_4$ because $G$ is a connected regular graph. In both cases, $G$ has at most three distinct eigenvalues.

We complete the proof.
\end{proof}

Recall that $\mathcal{G}(4,2)$ denotes the set of connected regular graphs with four distinct eigenvalues in which exactly two eigenvalues are simple. The following lemma provides a necessary  condition for the graphs belonging to $\mathcal{G}(4,2)$.

\begin{lem}\label{lem-3}
If $G$ is a connected $k$-regular graph on $n$ vertices with spectrum $\big\{[k]^1,[\lambda_2]^1,$ $[\lambda_3]^{m}, [\lambda_4]^{n-2-m}\big\}$ ($2\leq m\leq n-4$), then $G$ admits a  regular partition $V(G)=V_1\cup V_2$ into halves  with degrees $(\frac{1}{2}(k+\lambda_2),\frac{1}{2}(k-\lambda_2))$ such that
\begin{equation*}
|N_G(u) \cap N_G(v)|\!=\!\left\{
\begin{array}{ll}
\lambda_3\!+\!\lambda_4\!+\!\frac{1}{n}[(k\!-\!\lambda_3)(k\!-\!\lambda_4)\!+\!(\lambda_2\!-\!\lambda_3)(\lambda_2\!-\!\lambda_4)]&\mbox{if  $u\sim v$ in $V_1$ or $V_2$};\\
\frac{1}{n}[(k\!-\!\lambda_3)(k\!-\!\lambda_4)\!+\!(\lambda_2\!-\!\lambda_3)(\lambda_2\!-\!\lambda_4)]&\mbox{if  $u\nsim v$ in $V_1$ or $V_2$};\\
\lambda_3\!+\!\lambda_4\!+\!\frac{1}{n}[(k\!-\!\lambda_3)(k\!-\!\lambda_4)\!-\!(\lambda_2\!-\!\lambda_3)(\lambda_2\!-\!\lambda_4)]&\mbox{if  $u\sim v$, $u\in V_1$, $v\in V_2$};\\
\frac{1}{n}[(k\!-\!\lambda_3)(k\!-\!\lambda_4)\!-\!(\lambda_2\!-\!\lambda_3)(\lambda_2\!-\!\lambda_4)]&\mbox{if  $u\nsim v$, $u\in V_1$, $v\in V_2$}.
\end{array}
\right.
\end{equation*}
\end{lem}

\begin{proof}
Since $\lambda_2$ ($\neq k$) is a simple eigenvalue of $G$,  by Lemmas \ref{lem-5}, \ref{lem-2} and Corollary \ref{cor-1} we know that $n$ is even, $\lambda_2$ is an integer and $\mathbf{x}_2=\frac{1}{\sqrt{n}}(\mathbf{e}_\frac{n}{2}^T,-\mathbf{e}_\frac{n}{2}^T)^T$ is an eigenvector of $\lambda_2$. Putting $V_1=\{v\in V\mid \mathbf{x}_{2}(v)=1\}$ and $V_2=\{v\in V\mid \mathbf{x}_{2}(v)=-1\}$, again by Corollary \ref{cor-1} we  see that $V(G)=V_1\cup V_2$ is a regular  partition  of $G$ into halves with degrees $(\frac{1}{2}(k+\lambda_2),\frac{1}{2}(k-\lambda_2))$. Furthermore, the matrix $(A-\lambda_3I_n)(A-\lambda_4I_n)$ has the spectral decomposition
\begin{equation*}
(A-\lambda_3I_n)(A-\lambda_4I_n)=\frac{1}{n}(k-\lambda_3)(k-\lambda_4)\mathbf{e}_n\mathbf{e}_n^T+(\lambda_2-\lambda_3)(\lambda_2-\lambda_4)\mathbf{x}_2\mathbf{x}_2^T,
\end{equation*}
that is,
\begin{equation}\label{equ-5}
A^2-(\lambda_3+\lambda_4)A+\lambda_3\lambda_4I_n=\frac{1}{n}(k-\lambda_3)(k-\lambda_4)J_n+\frac{1}{n}(\lambda_2-\lambda_3)(\lambda_2-\lambda_4)
\left(\begin{matrix}
J_\frac{n}{2}&-J_\frac{n}{2}\\
-J_\frac{n}{2}&J_\frac{n}{2}
\end{matrix}\right).
\end{equation}
Note that the $(u,v)$-entry of $A^2$ is equal to $|N_G(u) \cap N_G(v)|$. Then (\ref{equ-5}) implies that
\begin{equation*}
|N_G(u) \cap N_G(v)|\!=\!\left\{
\begin{array}{ll}
\lambda_3\!+\!\lambda_4\!+\!\frac{1}{n}[(k\!-\!\lambda_3)(k\!-\!\lambda_4)\!+\!(\lambda_2\!-\!\lambda_3)(\lambda_2\!-\!\lambda_4)]&\mbox{if  $u\sim v$ in $V_1$ or $V_2$};\\
\frac{1}{n}[(k\!-\!\lambda_3)(k\!-\!\lambda_4)\!+\!(\lambda_2\!-\!\lambda_3)(\lambda_2\!-\!\lambda_4)]&\mbox{if  $u\nsim v$ in $V_1$ or $V_2$};\\
\lambda_3\!+\!\lambda_4\!+\!\frac{1}{n}[(k\!-\!\lambda_3)(k\!-\!\lambda_4)\!-\!(\lambda_2\!-\!\lambda_3)(\lambda_2\!-\!\lambda_4)]&\mbox{if  $u\sim v$, $u\in V_1$, $v\in V_2$};\\
\frac{1}{n}[(k\!-\!\lambda_3)(k\!-\!\lambda_4)\!-\!(\lambda_2\!-\!\lambda_3)(\lambda_2\!-\!\lambda_4)]&\mbox{if  $u\nsim v$, $u\in V_1$, $v\in V_2$}.
\end{array}
\right.
\end{equation*}

This completes the proof.
\end{proof}

The \emph{Kronecker product} $A\otimes B$ of matrices $A=(a_{ij})_{m\times n}$ and $B=(b_{ij})_{p\times q}$ is the $mp\times nq$ matrix obtained from $A$ by replacing each element $a_{ij}$ with the block $a_{ij}B$. Given a graph $G$ on  $n$  vertices with adjacency matrix $A$, we denote by $G\circledast J_m$ the graph with adjacency matrix $A\circledast J_m=(A+I_n)\otimes J_m-I_{nm}$ (see \cite{Dam1}).  By the definition, $G\circledast J_m$ is just the graph obtained from $G$ by replacing every vertex of $G$ with a clique $K_m$, and two such cliques are joined if and only if their corresponding vertices are adjacent in $G$. It is easy to see that the spectra of $G$ and $G\circledast J_m$ are determined by each other, that is,
\begin{equation}\label{equ-6}
\begin{array}{rl}
&\mathrm{Spec}(G)=\left\{[\lambda_1]^{m_1},\ldots,[\lambda_t]^{m_t}\right\}\\
\Leftrightarrow &\mathrm{Spec}(G\circledast J_m)=\big\{[m\lambda_1+m-1]^{m_1},\ldots,[m\lambda_t+m-1]^{m_t},[-1]^{nm-n}\big\}.
\end{array}
\end{equation}

Recall that $\mathcal{G}(4,2,-1)$ denotes the set of graphs belonging to $\mathcal{G}(4,2)$ with $-1$ as an  eigenvalue.  The following result gives a partial characterization of the graphs in $\mathcal{G}(4,2,-1)$.
\begin{thm}\label{thm-main-1}
Let $G\in \mathcal{G}(4,2,-1)$. Then $-1$ is a non-simple eigenvalue of $G$ if and only if  $G=K_{s,s}\circledast J_{t}$ with $s,t\geq 2$, or $G=K_{s,s}^{-}\circledast J_{t}$ with $s\geq 3$ and $t\geq 1$.
\end{thm}
\begin{proof}
By the assumption, let $G$ be a connected $k$-regular graph with $\mathrm{Spec}(G)=\big\{[k]^1,[\alpha]^1,$ $[\beta]^{n-2-m},[-1]^{m}\big\}$ ($2\leq m\leq n-4$). By considering the traces of $A(G)$ and $A^2(G)$, we get
\begin{equation}\label{equ-main-0}
\left\{\begin{aligned}
&k+\alpha+(n-2-m)\beta-m=0,\\
&k^2+\alpha^2+(n-2-m)\beta^2+m=kn.
\end{aligned}\right.
\end{equation}
Owing to  $\beta\neq -1$ and $k+\alpha+(n-2-m)\beta-m=0$, we have $n-k-\alpha-2\neq 0$. Then from (\ref{equ-main-0}) we can deduce that
\begin{equation}\label{equ-main-1}
\left\{\begin{aligned}
&\beta=\frac{kn-k^2-k-\alpha^2-\alpha}{n-k-\alpha-2},\\
&m=n-1+\frac{(n-k-1)(n-2\alpha-2)}{k^2 + (2 - n)k + \alpha^2 + 2\alpha - n + 2}.
\end{aligned}\right.
\end{equation}
Since $\alpha$ ($\neq k$) is a simple eigenvalue of $G$, by Lemma \ref{lem-3} we know that
$G$ admits a regular partition $V(G)=V_1\cup V_2$ into halves with degree $(\frac{1}{2}(k+\alpha),\frac{1}{2}(k-\alpha))$,  and that if $u,v\in V_i$ ($i=1,2$) are adjacent, then
\begin{equation}\label{equ-main-2}
|N_G(u) \cap N_G(v)|=\begin{aligned}
\beta-1+\frac{1}{n}[(k-\beta)(k+1)+(\alpha-\beta)(\alpha+1)].
\end{aligned}
\end{equation}
Combining (\ref{equ-main-1}) and (\ref{equ-main-2}),  we  have
$|N_G(u) \cap N_G(v)|=k-1$ by simple computation, so $u$ and $v$ have the same neighbors, that is, $N_G(u)\backslash \{v\}=N_G(v)\backslash \{u\}$. Since $u$ has exactly $\frac{k+\alpha}{2}$ neighbors in $G[V_i]$, each of them has the same neighbors (in $G$) as $u$, we see that such $\frac{k+\alpha}{2}$ neighbors together with $u$ induce a clique $K_{\frac{k+\alpha+2}{2}}$ which is totally included in $G[V_i]$ itself.  Furthermore, again by $N_G(u)\backslash \{v\}=N_G(v)\backslash \{u\}$,  there are no edges between any two such cliques in $V_i$, and if $v_1\in V_1$ is adjacent to $v_2\in V_2$, then the clique containing $v_1$ must be joined with the clique containing $v_2$.  Moreover, both $G[V_1]$ and $G[V_2]$ consist of the disjoint union of $\frac{n}{k+\alpha+2}$ copies of $K_{\frac{k+\alpha+2}{2}}$ because $|V_1|=|V_2|=\frac{n}{2}$.
Hence, there exists a regular bipartite graph $H$ on $\frac{2n}{k+\alpha+2}$ vertices  such that $G=H\circledast J_{\frac{k+\alpha+2}{2}}$. Then from (\ref{equ-6}) we  deduce that
\begin{equation*}
\mathrm{Spec}(H)=\Bigg\{\bigg[\frac{k-\alpha}{k+\alpha+2}\bigg]^1,
\bigg[-\frac{k-\alpha}{k+\alpha+2}\bigg]^1,
\bigg[\frac{2\beta-k-\alpha}{k+\alpha+2}\bigg]^{n-m-2},
[-1]^{m-\frac{k+\alpha}{k+\alpha+2}n}\Bigg\}.
\end{equation*}
Since $\frac{k-\alpha}{k+\alpha+2}$ is the  maximum eigenvalue of $H$ which is simple, $H$ must be a connected $\big(\frac{k-\alpha}{k+\alpha+2}\big)$-regular bipartite graph. Clearly,  $-\frac{k-\alpha}{k+\alpha+2}\neq -1$ since otherwise we have $\alpha=-1$, a contradiction.  Thus $-\frac{k-\alpha}{k+\alpha+2}<-1$ because  it is an integer, so $\alpha<-1$. We consider the following two situations.

\textbf{Case 1.} $m-\frac{k+\alpha}{k+\alpha+2}n=0$;

Since $H$ is  bipartite,  we get $\frac{2\beta-k-\alpha}{k+\alpha+2}=0$, and then $\beta=\frac{k+\alpha}{2}$. Putting $\beta=\frac{k+\alpha}{2}$ in (\ref{equ-main-1}) and  considering  $\alpha<k$,  we  deduce that $\alpha=k-n$, $\beta=k-\frac{n}{2}$, $m=\frac{2kn-n^2}{2k-n+2}$, and $\frac{n}{2}+1\leq k\leq\frac{3n}{4}-1$ because $\alpha\geq-k$ and $2\leq m\leq n-4$. Thus
\begin{equation*}
\mathrm{Spec}(H)=\Bigg\{\bigg[\frac{n}{2k-n+2}\bigg]^1,\bigg[-\frac{n}{2k-n+2}\bigg]^1,
[0]^{\frac{2n}{2k-n+2}-2}\Bigg\}.
\end{equation*}
Then $H$ is a connected $\big(\frac{n}{2k-n+2}\big)$-regular bipartite graph with three distinct eigenvalues. Therefore, $H=K_{\frac{n}{2k-n+2},\frac{n}{2k-n+2}}$ because  complete bipartite graphs are the only connected  bipartite graphs with  three distinct eigenvalues. Hence, $G=H\circledast J_{\frac{k+\alpha+2}{2}}= K_{\frac{n}{2k-n+2},\frac{n}{2k-n+2}}\circledast J_{\frac{2k-n+2}{2}}$. If we set $s=\frac{n}{2k-n+2}$ and $t=\frac{2k-n+2}{2}$, then $G=K_{s,s} \circledast J_{t}$ with $s,t\geq 2$ from the above arguments.

\textbf{Case 2.} $m-\frac{k+\alpha}{k+\alpha+2}n\geq 1$.

Since $H$ is bipartite, we have $\frac{2\beta-k-\alpha}{k+\alpha+2}=-(-1)=1$, and so $\beta=k+\alpha+1$. Combining  this with (\ref{equ-main-1}) we obtain that
$(\alpha+1)(n-2k-2)=0$, which implies that $k=\frac{n}{2}-1$ because $\alpha\neq -1$. Then $\beta=\frac{n}{2}+\alpha$, $m=n-1-\frac{2n}{n+2\alpha+2}$ by  (\ref{equ-main-1}), and so $-\frac{n}{2}+1\leq\alpha\leq-\frac{n}{6}-1$ because $2\leq m\leq n-4$. Furthermore, $n-m-2=m-\frac{k+\alpha}{k+\alpha+2}n=\frac{n-2\alpha-2}{n+2\alpha+2}$.  Thus we get
\begin{equation*}
\mathrm{Spec}(H)=\Bigg\{\bigg[\frac{n-2\alpha-2}{n+2\alpha+2}\bigg]^1,\bigg[-\frac{n-2\alpha-2}{n+2\alpha+2}\bigg]^1,[1]^{\frac{n-2\alpha-2}{n+2\alpha+2}},[-1]^{\frac{n-2\alpha-2}{n+2\alpha+2}}\Bigg\}.
\end{equation*}
Then $H$ is a connected $\big(\frac{n-2\alpha-2}{n+2\alpha+2}\big)$-regular bipartite graph with four distinct eigenvalues. By Lemma \ref{lem-4}, we may conclude that $H$ is the incidence graph of  a symmetric BIBD with parameters $\big(\frac{n-2\alpha-2}{n+2\alpha+2}+1,\frac{n-2\alpha-2}{n+2\alpha+2},\frac{n-2\alpha-2}{n+2\alpha+2}-1\big)$. It is well known that such a BIBD is unique (see \cite{Dam1}), and the corresponding incidence  graph can be obtained by removing a perfect matching from the  complete bipartite graph $K_{s,s}$ , where $s=\frac{n-2\alpha-2}{n+2\alpha+2}+1=\frac{2n}{n+2\alpha+2}$. Hence, $G=H\circledast J_{\frac{k+\alpha+2}{2}}= K_{\frac{2n}{n+2\alpha+2},\frac{2n}{n+2\alpha+2}}^{-}\circledast J_{\frac{n+2\alpha+2}{4}}$. If we put $s=\frac{2n}{n+2\alpha+2}$ and $t=\frac{n+2\alpha+2}{4}$, then $G=K_{s,s}^{-}\circledast J_t$ with $s\geq 3$ and $t\geq 1$ from the above arguments.

Conversely,  by simple computation we obtain that
\begin{equation}\label{equ-main-3}
\begin{aligned}
&\mathrm{Spec}(K_{s,s}\circledast J_t)=\{[st+t-1]^1,[-st+t-1]^1,[t-1]^{2s-2},[-1]^{2s(t-1)}\};\\
&\mathrm{Spec}(K_{s,s}^{-}\circledast J_t)=\{[st-1]^1,[-st+2t-1]^1,[2t-1]^{s-1},[-1]^{2st-s-1}\},
\end{aligned}
\end{equation}
and our result follows.
\end{proof}

By Theorem \ref{thm-main-1}, we obtain the following corollary immediately.
\begin{cor}\label{cor-2}
The graphs $K_{s,s}\circledast J_{t}$ ($s,t\geq 2$) and $K_{s,s}^{-}\circledast J_{t}$ ($s\geq 3$, $t\geq 1$) are DS.
\end{cor}
\begin{proof}
Note that any graph cospectral with $K_{s,s}\circledast J_{t}$  or $K_{s,s}^{-}\circledast J_{t}$ must be connected. By Theorem \ref{thm-main-1}, it suffices to prove that $K_{s_1,s_1}\circledast J_{t_1}$ ($s_1,t_1\geq 2$) and $K_{s_2,s_2}^{-}\circledast J_{t_2}$ ($s_2\geq 3$, $t_2\geq 1$) cannot share the same spectrum. On the contrary, assume that  $\mathrm{Spec}(K_{s_1,s_1}\circledast J_{t_1})=\mathrm{Spec}(K_{s_2,s_2}^-\circledast J_{t_2})$. Then  $s_1t_1=s_2t_2$ and $s_1t_1+t_1-1=s_2t_2-1$ because they are regular graphs which have the same number of vertices and share the common degree. This implies that $t_1=0$, which is impossible because $t_1\geq 2$.
\end{proof}
\begin{remark}\label{rem-1}
\emph{It is worth mentioning that the DS-properties of $K_{s,s}\circledast J_{t}$ and $K_{s,s}^{-}\circledast J_{t}$ have been mentioned in  \cite{Dam6} and \cite{Dam1}, respectively.}
\end{remark}
Recall that $\mathcal{G}(4,\geq -1)$ denotes the set of connected regular graphs with four distinct eigenvalues and second least eigenvalue  not less than $-1$.  Petrovi\'{c} in \cite{Petrovic} determined all the connected graphs whose second least eigenvalue is not less than $-1$. However, it is a difficult work to pick out all the graphs belonging to $\mathcal{G}(4,\geq -1)$ from their characterization. Here, from Theorem \ref{thm-main-1} and Lemma \ref{lem-6}, we can easily give a complete characterization of the graphs in $\mathcal{G}(4,\geq -1)$.
\begin{thm}\label{thm-main-2}
A connected graph $G\in \mathcal{G}(4,\geq -1)$  if and only if $G=K_{s,s}\circledast J_{t}$ with $s,t\geq 2$, or $G=K_{s,s}^{-}\circledast J_{t}$ with $s\geq 3$ and $t\geq 1$.
\end{thm}
\begin{proof}
Suppose that $G\in \mathcal{G}(4,\geq -1)$, and  $\alpha$, $\beta$ are the least and second least eigenvalues of $G$, respectively.  If $\beta>-1$, then  $G=K_{m,n}$, $A(l,m,n)$ or $B(l,m,n,p)$ with proper parameters $l,m,n,p$ by Lemma \ref{lem-6}. It is seen that $A(l,m,n)$ and  $B(l,m,n,p)$ cannot be regular, so $G=K_{n,n}$ with $n\geq 2$. Note that $K_{n,n}$ has  only three distinct eigenvalues,  so there are no graphs in $\mathcal{G}(4,\geq -1)$ with  $\beta>-1$, and thus $\beta=-1$ because $G\in \mathcal{G}(4,\geq -1)$. We claim that $\alpha<\beta=-1$ since otherwise $G$ will be a complete graph due to the least eigenvalue of $G$ is  $-1$, and thus $\alpha$ is a simple eigenvalue of $G$. As a consequence, $G$ will belong to $\mathcal{G}(4,2,-1)$ and $-1$ is a non-simple eigenvalue of $G$  by Theorem \ref{thm-main-0}. Hence, $G=K_{s,s}\circledast J_{t}$ with $s,t\geq 2$, or $G=K_{s,s}^{-}\circledast J_{t}$ with $s\geq 3$ and $t\geq 1$ by Theorem \ref{thm-main-1}.

Conversely, from (\ref{equ-main-3}) we obtain the required result immediately.
\end{proof}

Now we continue to  consider the graphs in $\mathcal{G}(4,2,-1)$. The following result excludes the existence of such graphs belonging to (2b) (see Section \ref{s-1}).
\begin{thm}\label{thm-main-3}
There are no connected $k$-regular graphs with spectrum $\{[k]^1,[-1]^1,[\alpha]^{m},$ $[\beta]^{n-2-m}\}$, where $\alpha$ and $\beta$ are not integers and $2\leq m\leq n-4$.
\end{thm}
\begin{proof}
On the conrary, assume that $G$ is such a graph.  Then $G$ will be a connected $k$-regular graphs having exactly two integral eigenvalues. By Lemma \ref{lem-1}, the  eigenvalues $\alpha$ and $\beta$ are of the form $\frac{1}{2}(a\pm\sqrt{b})$ ($a,b\in \mathbb{Z}$, $b>0$) and have the same multiplicity, i.e., $m=\frac{1}{2}(n-2)$. Then $\alpha$ and $\beta$ satisfy  the following two equations:
\begin{equation*}
\left\{\begin{aligned}
&k-1+\frac{1}{2}(n-2)\alpha+\frac{1}{2}(n-2)\beta=0,\\
&k^2+1+\frac{1}{2}(n-2)\alpha^2+\frac{1}{2}(n-2)\beta^2=kn.
\end{aligned}\right.
\end{equation*}
By simple computation, we obtain
\begin{equation*}
\left\{\begin{aligned}
&\alpha=\frac{-k+1+\sqrt{(kn-k-1)(n-k-1)}}{n-2},\\
&\beta=\frac{-k+1-\sqrt{(kn-k-1)(n-k-1)}}{n-2}.
\end{aligned}\right.
\end{equation*}
Considering that $\alpha$ and $\beta$ are of the form $\frac{1}{2}(a\pm\sqrt{b})$ ($a,b\in \mathbb{Z}$), we claim that  $\frac{-k+1}{n-2}=\frac{a}{2}$, and so  $a=-1$ because $a\in \mathbb{Z}$ and $1<k<n-1$. Thus $n=2k$, and then
\begin{equation}\label{equ-main-4}
\begin{aligned}
\alpha=\frac{1}{2}(-1+\sqrt{2k + 1}),~~~~\beta=\frac{1}{2}(-1-\sqrt{2k + 1}).
\end{aligned}
\end{equation}
Furthermore,  by Lemma \ref{lem-3}, the graph $G$ admits a  regular partition $V(G)=V_1\cup V_2$ into halves  with degree $(\frac{1}{2}(k-1),\frac{1}{2}(k+1))$ such that
\begin{equation*}
|N_G(u) \cap N_G(v)|=\left\{
\begin{aligned}
&\alpha+\beta+\frac{1}{n}[(k-\alpha)(k-\beta)+(-1-\alpha)(-1-\beta)]&\mbox{if  $u\sim v$ in $V_i$};\\
&\frac{1}{n}[(k-\alpha)(k-\beta)+(-1-\alpha)(-1-\beta)]&\mbox{if  $u\nsim v$ in $V_i$}.
\end{aligned}
\right.
\end{equation*}
Putting (\ref{equ-main-4}) and $n=2k$  in the above equation, we get
\begin{equation*}
|N_G(u) \cap N_G(v)|=\left\{
\begin{aligned}
&\frac{k}{2}-1&\mbox{if  $u\sim v$ in $V_i$};\\
&\frac{k}{2}&\mbox{if  $u\nsim v$ in $V_i$},
\end{aligned}
\right.
\end{equation*}
which is impossible because $k$ is  odd  due to $\frac{1}{2}(k-1)$ is an integer.

This completes the proof.
\end{proof}

By Theorems \ref{thm-main-1} and  \ref{thm-main-3}, in order to determine all the graphs in  $\mathcal{G}(4,2,-1)$, it remains to consider such graphs belonging to (2a) (see Section \ref{s-1}), i.e., the $k$-regular graphs with spectrum $\{[k]^1,[-1]^1,[\alpha]^{m},[\beta]^{n-2-m}\}$, where $\alpha$ and $\beta$ are  integers and  $2\leq m\leq n-4$.  By Appendix A in \cite{Dam2}, we find that there are no such graphs up to $30$ vertices. In what follows, we will prove the non-existence of such graphs for arbitrary $n$.

\begin{thm}\label{thm-main-4}
There are no connected $k$-regular graphs with spectrum $\{[k]^1,[-1]^1,[\alpha]^{m},$ $[\beta]^{n-2-m}\}$, where $\alpha$ and $\beta$ are integers and $2\leq m\leq n-4$.
\end{thm}
\begin{proof}
On the contrary, assume that $G$ is such a graph with adjacency matrix $A$.  Firstly, we assert that $3\leq k\leq n-2$. In fact,  if $k=2$, then $G=C_n$, where $n$ is even by Lemmas \ref{lem-5} and \ref{lem-2}, and so  $G$ is a bipartite graph, which is impossible because $2\leq m\leq n-4$; if $k=n-1$, then $G$ is a complete graph, which has exactly two distinct eigenvalues, a contradiction.

Now suppose that $\alpha\beta\geq 0$, then $\alpha,\beta\geq 0$ or $\alpha,\beta\leq 0$. In the former case, we see that $G$ is a complete graph because $-1$ is the least eigenvalue of $G$, which is a contradiction. In the later case, we claim that $G$ is a  complete multipartite graph because $G$ has only one positive eigenvalue, thus $G=K_{n-k,n-k,\ldots,n-k}$ due to $G$ is $k$-regular. Then $\mathrm{Spec}(G)=\big\{[k]^1,[n-k-2]^{\frac{n}{n-k}-1},[0]^{n-\frac{n}{n-k}}\big\}$, a contradiction. Thus $\alpha\beta<0$. Without loss of generality, we may assume that $\alpha\geq 1$ and $\beta\leq -2$. By considering the traces of $A$ and $A^2$, we get
\begin{equation*}
\left\{\begin{aligned}
&k-1+m\alpha+(n-2-m)\beta=0,\\
&k^2+1+m\alpha^2+(n-2-m)\beta^2=kn.
\end{aligned}\right.
\end{equation*}
Canceling out $m$ in the above two equations, we have
\begin{equation}\label{equ-main-5}
\beta=-\frac{k(n-k)+(k-1)\alpha-1}{(n-2)\alpha+k-1}.
\end{equation}
By Lemmas \ref{lem-5} and \ref{lem-2}, it is seen that $n$ is a divisor of $(k-\alpha)(k-\beta)+(-1-\alpha)(-1-\beta)$ and $(k-\alpha)(k-\beta)-(-1-\alpha)(-1-\beta)$, and so  a divisor of $-2(-1-\alpha)(-1-\beta)=-2(1+\alpha)(1+\beta)>0$ because $\alpha\geq 1$ and $\beta\leq -2$ are integers. Thus  $c=\frac{-2(1+\alpha)(1+\beta)}{n}$ is a positive integer.  Combining this with (\ref{equ-main-5}), we get
\begin{eqnarray*}
c&=&\frac{2(k-\alpha)(\alpha+1)(n-k-1)}{n(n-2)\alpha+n(k-1)}
<\frac{2(k-\alpha)(\alpha+1)(n-k-1)}{n(n-2)\alpha}=\frac{2(n-k-1)}{n\cdot\frac{n-2}{k-\alpha}\cdot\frac{\alpha}{\alpha+1}}\\
&<&\frac{4(n-k-1)}{n}<4.
\end{eqnarray*}
It suffices to consider the following three situations.

\textbf{Case 1.} $c=1$;

In this case, we get $\frac{2(k-\alpha)(\alpha+1)(n-k-1)}{(n-2)\alpha+k-1}=n$, that is,
\begin{equation}\label{equ-main-6}
2(n-k-1)\alpha^2+(n^2-2kn+2k^2-2)\alpha+(2k-n)(k+1)=0,
\end{equation}
which implies that $n>2k$ because $n-k-1>0$, $n^2-2kn+2k^2-2>0$ and $\alpha\geq 1$. Solving  (\ref{equ-main-6}), we get
\begin{equation}\label{equ-main-7}
\alpha=\frac{-(n^2\!-\!2kn\!+\!2k^2\!-\!2)+\!\sqrt{(n^2-2kn+2k^2-2)^2+8(n-k-1)(n-2k)(k+1)}}{4(n-k-1)}.
\end{equation}
Again from (\ref{equ-main-5}) we obtain
\begin{equation}\label{equ-main-8}
\beta=\frac{-(n^2\!-\!2kn\!+\!2k^2\!-\!2)-\!\sqrt{(n^2-2kn+2k^2-2)^2+8(n-k-1)(n-2k)(k+1)}}{4(n-k-1)}.
\end{equation}
Then, by Lemma \ref{lem-3}, $G$ admits a  regular partition $V(G)=V_1\cup V_2$ into halves  with degrees $(\frac{1}{2}(k-1),\frac{1}{2}(k+1))$ such that
\begin{equation}\label{equ-main-9}
|N_G(u) \cap N_G(v)|=\alpha+\beta+\frac{1}{n}[(k-\alpha)(k-\beta)+(-1-\alpha)(-1-\beta)]\mbox{ for   $u\sim v$ in $V_i$}.
\end{equation}
Combining (\ref{equ-main-7}), (\ref{equ-main-8}) and (\ref{equ-main-9}), we thus have $|N_G(u)\cap N_G(v)|=k-\frac{n}{2}-1$. This implies that
$n\leq 2k-2$ because $|N_G(u)\cap N_G(v)|\geq 0$, contrary to $n>2k$.

\textbf{Case 2.} $c=2$;

In this case, we have $\frac{(k-\alpha)(\alpha+1)(n-k-1)}{(n-2)\alpha+k-1}=n$. This implies that
\begin{equation*}
\left\{\begin{aligned}
&\alpha=\frac{-(n^2\!-\!(k\!+\!1)n\!+\!k^2\!-\!1)+\!\sqrt{(n^2\!-\!(k\!+\!1)n\!+\!k^2\!-\!1)^2\!-\!4(n\!-\!k\!-\!1)(k^2\!+\!k\!-\!n)}}{2(n-k-1)},\\
&\beta=\frac{-(n^2\!-\!(k\!+\!1)n\!+\!k^2\!-\!1)-\!\sqrt{(n^2\!-\!(k\!+\!1)n\!+\!k^2\!-\!1)^2\!-\!4(n\!-\!k\!-\!1)(k^2\!+\!k\!-\!n)}}{2(n-k-1)}.
\end{aligned}\right.
\end{equation*}
Again from (\ref{equ-main-9}) we deduce that $|N_G(u)\cap N_G(v)|=k-n-1<0$ for $u\sim v$ in $V_i$ ($i=1,2$), which is a contradiction.

\textbf{Case 3.} $c=3$.

In this case, we have $\frac{2(k-\alpha)(\alpha+1)(n-k-1)}{(n-2)\alpha+k-1}=3n$, that is,
\begin{equation*}
2(n-k-1)\alpha^2+[3n^2-(2k+4)n+2k^2-2]\alpha+(k-3)n+2k^2+2k=0,
\end{equation*}
which is  impossible because $n-k-1>0$, $3n^2-(2k+4)n+2k^2-2>0$, $(k-3)n+2k^2+2k>0$ due to $k\geq 3$, and  $\alpha >0$.

We complete this proof.
\end{proof}

Combining Theorems \ref{thm-main-1}, \ref{thm-main-3} and \ref{thm-main-4}, we obtain the  main result of this paper.
\begin{thm}\label{thm-main-5}
A connected graph $G\in \mathcal{G}(4,2,-1)$ if and only if $G=K_{s,s}\circledast J_{t}$ with $s,t\geq 2$, or $G=K_{s,s}^{-}\circledast J_{t}$ with $s\geq 3$ and $t\geq 1$.
\end{thm}
Recall that  $\mathcal{G}(4,2,0)$ denotes the set of  graphs belonging to $\mathcal{G}(4,2)$ with $0$ as an  eigenvalue. Since the spectrum of a regular graph could be deduced from its complement,  we can easily characterize all the graphs in $\mathcal{G}(4,2,0)$ by Theorems \ref{thm-main-1}, \ref{thm-main-3} and \ref{thm-main-4}.
\begin{thm}\label{thm-main-6}
A connected graph $G\in \mathcal{G}(4,2,0)$ if and only if $G=\overline{K_{s,s}^{-}\circledast J_{t}}$ with $s\geq 3$ and $t\geq 1$.
\end{thm}
\begin{proof}
 Let $G\in\mathcal{G}(4,2,0)$ be a connected $k$-regular graph with adjacency matrix $A$, and let $\overline{G}$ be the complement of $G$.

If $0$ is a non-simple eigenvalue of $G$, suppose that $\mathrm{Spec}(G)=\big\{[k]^1,[\alpha]^1,[\beta]^{n-2-m},[0]^{m}\big\}$ ($2\leq m\leq n-4$). Then $\overline{G}$ is a $(n-k-1)$-regular graph with $\mathrm{Spec}(\overline{G})=\big\{[n-k-1]^1,[-1-\alpha]^1,[-1-\beta]^{n-2-m},[-1]^{m}\big\}$.
We claim that $\alpha>0$ and $\beta<0$ or $\alpha<0$ and $\beta>0$, since otherwise $G$ will be  a regular complete multipartite graph (which  has only three distinct eigenvalues) or does not exist. Assume that $\overline{G}$ is  disconnected.  If $\alpha>0$ and $\beta<0$,  we have $n-k-1=-1-\beta$, i.e., $\beta=k-n$. By considering the traces of $A$ and $A^2$, we obtain
\begin{equation*}
\left\{\begin{aligned}
&k+\alpha+(n-2-m)(k-n)=0,\\
&k^2+\alpha^2+(n-2-m)(k-n)^2=kn,\\
\end{aligned}\right.
\end{equation*}
and so $\alpha=0$ or $\alpha=k-n$, which are impossible due to $\alpha>0$. If $\alpha<0$ and $\beta>0$, similarly,  we have $\alpha=k-n$ because $\overline{G}$ is disconnected and regular, and so $n=2k$, or $n\neq 2k$ and $\beta=k-n$ by considering the traces of $A$ and $A^2$. In both cases, we can deduce a contradiction  because $G$ cannot be a bipartite graph and $\alpha\neq\beta$. Therefore, $\overline{G}$ must be connected, and thus $\overline{G}\in\mathcal{G}(4,2,-1)$ with $-1$ as a non-simple eigenvalue. By Theorem \ref{thm-main-1}, we may conclude that $\overline{G}=K_{s,s}\circledast J_{t}$ with $s,t\geq 2$, or $\overline{G}=K_{s,s}^{-}\circledast J_{t}$ with $s\geq 3$ and $t\geq 1$. Hence, $G=\overline{K_{s,s}^{-}\circledast J_{t}}$ with $s\geq 3$ and $t\geq 1$ because $G$ is connected.

If $0$ is a simple eigenvalue of $G$, we suppose that $\mathrm{Spec}(G)=\big\{[k]^1,[0]^1,[\alpha]^{m},[\beta]^{n-2-m}\big\}$. Then $\mathrm{Spec}(\overline{G})=\big\{[n-k-1]^1,[-1]^1,[-1-\alpha]^{m},[-1-\beta]^{n-2-m}\big\}$. As above, one can easily deduce that $\overline{G}$ is connected, and so $\overline{G}\in\mathcal{G}(4,2,-1)$ with $-1$ as a simple eigenvalue. Then, by Theorems \ref{thm-main-3} and \ref{thm-main-4}, $\overline{G}$ does not exist and so is $G$.

Consequently, if $G\in\mathcal{G}(4,2,0)$ then $G=\overline{K_{s,s}^{-}\circledast J_{t}}$ with $s\geq 3$ and $t\geq 1$. Obviously, $\overline{K_{s,s}^{-}\circledast J_{t}}\in \mathcal{G}(4,2,0)$ because $\mathrm{Spec}(\overline{K_{s,s}^{-}\circledast J_{t}})=\{[st]^1,[st-2t]^1,[-2t]^{s-1},[0]^{2st-s-1}\}$. Our result follows.
\end{proof}

\section*{Acknowledgements}
The authors are grateful to Professor Richard A. Brualdi  and the referees for their valuable comments and  helpful suggestions, which have considerably improved the presentation of this paper.

\end{document}